\documentclass[11pt,a4paper,twoside]{article}
\usepackage{amssymb}
\usepackage{amsmath}
\usepackage{fancyhdr}
\usepackage{amsthm}
\usepackage{rotate}
\usepackage{graphicx}
\usepackage[T1]{fontenc}
\usepackage{afterpage}  %

\usepackage{hyperref}

\newtheorem{theorem}{Theorem}
\newtheorem{proposition}{Proposition}

\newtheorem{lemma}{Lemma}
\newtheorem{corollary}{Corollary}
\newtheorem{question}{Question}

\DeclareMathOperator{\spec}{Spec}
\DeclareMathOperator{\cent}{Cent}

\newcommand{\bnum}{\begin{enumerate}}
\newcommand{\enum}{\end{enumerate}}

\begin{document}

\afterpage{\rhead[]{\thepage} \chead[\small R. Sharafdini  and R. Darbandi]{\small  } \lhead[\thepage]{} }                  

\begin{center}
\vspace*{2pt}
{\Large \textbf{Energy of commuting graph of finite groups}}\\[30pt]
{\large \textsf{\emph{Reza Sharafdini $^\ast$, Rajat Kanti Nath,\quad Rezvan Darbandi  }}}
\\[30pt]
\end{center}
\textbf{Abstract.}{\footnotesize
Let $\Gamma$ be a graph with the adjacency matrix $A$.
The energy of  $\Gamma$ is the sum of the absolute values of the eigenvalues of  $A$. In this article we  compute the  energies of the commuting graphs of some finite groups and discuss some consequences.
}



\footnote{$^\ast$ Corresponding author}

\section*{\centerline{1. Introduction}}\setcounter{section}{1}

 The commuting graph of a non-Abelian group $G$ with center $Z(G)$ is an undirected graph with vertex set $G\setminus Z(G)$, and two distinct vertices $a$ and $b$ are adjacent whenever $ab = ba$. We write  $\Gamma_G$ to denote this graph.  Various properties of $\Gamma_G$ have been studied  in \cite{amr06,{Bat-bhR09},iJ07,iJ08,mP13,par13}.
It is noted in \cite{Ne} that the complement of  $\Gamma_G$, known as non-comuting graph of $G$ (see \cite{Ab06}), was first considered by Erd\"{o}s in 1975.



Let $g$ be any element of a group $G$. Then the subgroup  $C_G(g) = \{a \in G \mid  ag = ga\}$ is called centralizer of  $g$. If $C_G(g)$ is Abelian for every $g \in G\setminus Z(G)$ then the group $G$ is called an AC-group. If $\cent(G) = \{C_G(g) \mid  g \in G\}$ and $|\cent(G)| = n$ then  $G$ is called an $n$-centralizer group (see \cite{BelG94,Ash1}).





Let $\Gamma$ be a graph with the adjacency matrix $A$. Let
$\lambda_1, \lambda_2, \ldots, \lambda_t$ be the distinct eigenvalues of
$A$ with corresponding mutiplicites $m_1, m_2, \ldots, m_t$. The spectrum of $\Gamma$ is defined as
\[
\spec(\Gamma)=\Big\{\lambda_1^{m_1}, \lambda_2^{m_2}, \ldots, \lambda_t^{m_t}\Big\}.
\]
The energy of  $\Gamma$ (see \cite{Gu78,LiShiGut2012}) is given by
\begin{equation}\label{Energy-eq}
\mathcal{E}(\Gamma) = \sum_{i = 1}^t m_i|\lambda_i|.
\end{equation}
\noindent It may be mentioned here that
$\mathcal{E}(K_n) = 2n - 2$, where $K_n$ is the complete graph on $n$ vertices.
Spectral aspects of $\Gamma_G$ are considered in \cite{Dut, Dut-integral,KJM-18,IJPAM-8,JLTA-18,IJEST-18} recently for various families of finite groups.
In this paper we  compute   energy of $\Gamma_G$ for several classes of finite AC-groups and discuss some consequences. Throughout the paper same notations are used to denote various families of  groups as in \cite{Dut,Dut-integral}.




\section*{\centerline{2. Main Results}}\setcounter{section}{2}


In this section, we compute the energy of $\Gamma_G$ for  some particular families of finite non-Abelian groups. 
In this section, we compute the energy of $\Gamma_G$ for  some particular families of finite non-Abelian groups.
Using spectrum related results available in \cite{Dut,Dut-integral} and  \eqref{Energy-eq} we get the following results.
\begin{theorem}\label{cor2}
We have
\begin{enumerate}
\item
$
\mathcal{E}(\Gamma_{ M_{2mn}}) =
\begin{cases}
&\!\!\!\!\!\!\!\!4mn-2m-2n-2  ~~\text{ if $m$ is odd}\\[3mm]
&\!\!\!\!\!\!\!\!4mn-4n-m-2 ~~\text{ if $m$ is even}.
\end{cases}
$

\item
$
\mathcal{E}(\Gamma_{D_{2m}}) =
\begin{cases}
2m-4 &\text{ if $m$ is odd}\\[3mm]
3m-6 &\text{ if $m$ is even}.
\end{cases}
$

\item
$
\mathcal{E}(\Gamma_{ Q_{4m}}) = 6m-3.
$

\item $\mathcal{E}(\Gamma_{U_{6n}}) = 10n-8.$
\end{enumerate}
\end{theorem}
\begin{theorem}\label{cor3}
Let $G$ be a finite group. Then we have the following.
\begin{enumerate}
\item If $G = QD_{2^n}$ then
$
\mathcal{E}(\Gamma_{G})
=2^{n}+2^{n - 1}- 6.
$

\item If $G = PSL(2, 2^k)$ then $\mathcal{E}(\Gamma_{G}) = 2^{3k + 1} - 2^{2k + 1} - 2.2^{k + 1} - 4$.

\item If $G = GL(2, q)$ then $\mathcal{E}(\Gamma_{G}) = 2q^4 - 2q^3 - 4q^2 - 2q $.
\item If $\frac{G}{Z(G)} \cong Sz(2)$ then
   $\mathcal{E}(\Gamma_{G})=28|Z(G)|-12.$
\item If $G = A(n, \vartheta)$ then
  $\mathcal{E}(\Gamma_{G})=2{(2^n - 1)^2}.$
\item If $G = A(n, p)$ then
  \[\mathcal{E}(\Gamma_{G})=({p^{3n} -2p^{n}-1})+({p^n + 1})(p^{2n} - p^n - 1)=2p^{3n}-4p^n-2.\]
\item If $G$ is a non-Abelian    and $|G|$ is product of two primes $p, q$  such that $p\mid (q - 1)$ then
$
\mathcal{E}(\Gamma_G)=2q(p-1)-3.
$
\end{enumerate}
\end{theorem}


In the following theorem we compute energy of commuting graphs of the groups $G$ such that $\frac{G}{Z(G)}$ is isomorphic to ${\mathbb{Z}}_p \times {\mathbb{Z}}_p$ and $D_{2m}$.
\begin{theorem}\label{cor:pp}
Let $G$ be a finite group.
\begin{enumerate}
  \item  If $p$ is a prime and $\frac{G}{Z(G)} \cong {\mathbb{Z}}_p \times {\mathbb{Z}}_p$, then $\mathcal{E}(\Gamma_G)=2((p^2 - 1)|Z(G)| -p - 1)$.
  \item  If $m \,(\geq 2)$ is a natural number and $\frac{G}{Z(G)} \cong D_{2m}$, then  $\mathcal{E}(\Gamma_G)=2((2m - 1)|Z(G)| -m - 1).$
\end{enumerate}
\end{theorem}
As applications of Theorem \ref{cor:pp} we also have the following two corollaries.
\begin{corollary}
Let $p$ be a prime.
\begin{enumerate}
\item If $G$ is a non-Abelian group and $|G| = p^3$,  then $\mathcal{E}(\Gamma_G)=2(p^3 -2p - 1).$
\item If $G$ is a  $4$-centralizer finite group, then $\mathcal{E}(\Gamma_G)=6(|Z(G)|-1)$.
\item
If $G$ is a  $(p+2)$-centralizer finite $p$-group,  then
$$\mathcal{E}(\Gamma_G)=2((p^2 - 1)|Z(G)| -p - 1).$$

\item If $G$ is a  $5$-centralizer finite  group, then $$\mathcal{E}(\Gamma_G)\in \Big\{ 8(2|Z(G)|-1),10|Z(G)|-8 \Big\}.$$
\end{enumerate}
\end{corollary}

\begin{proof}
 (i) In this case $|Z(G)| = p$ and  $\frac{G}{Z(G)} \cong {\mathbb{Z}}_p \times {\mathbb{Z}}_p$. Hence the  result follows from Theorem \ref{cor:pp}(i).\\
 (ii) The  result follows from Theorem \ref{cor:pp}(i), noting that
 $\frac{G}{Z(G)} \cong {\mathbb{Z}}_2 \times {\mathbb{Z}}_2$ (see \cite[Theorem 2]{BelG94}).\\
(iii) 
We have   $\frac{G}{Z(G)} \cong {\mathbb{Z}}_p \times {\mathbb{Z}}_p$ (see \cite[Lemma 2.7]{Ash1}). Hence the  result follows from Theorem \ref{cor:pp}(i).\\
(iv) 
By \cite[Theorem 4]{BelG94} we have  $\frac{G}{Z(G)} \cong {\mathbb{Z}}_3 \times {\mathbb{Z}}_3$ or $D_6$. If $\frac{G}{Z(G)} \cong {\mathbb{Z}}_3 \times {\mathbb{Z}}_3$, then by Theorem \ref{cor:pp}(i) we have $\mathcal{E}(\Gamma_G)=8(2|Z(G)|-1)$.
In the case that $\frac{G}{Z(G)} \cong D_6$, by Theorem \ref{cor:pp}(ii) we have
$\mathcal{E}(\Gamma_G)=10|Z(G)|-8$.
\end{proof}

\begin{corollary}\label{order16}
If a group $G$ is isomorphic to any of the following groups
\begin{enumerate}
\item ${\mathbb{Z}}_2 \times D_8$
\item ${\mathbb{Z}}_2 \times Q_8$
\item $M_{16}  = \langle x , y \mid  x^8 = y^2 = 1, yxy = x^5 \rangle$
\item ${\mathbb{Z}}_4 \rtimes {\mathbb{Z}}_4 = \langle x, y \mid  x^4 = y^4 = 1, yxy^{-1} = x^{-1} \rangle$
\item $D_8 * {\mathbb{Z}}_4 = \langle x, y, c \mid  x^4 = y^2 = c^2 =  1, xy = yx, xc = cx, yc = x^2cy \rangle$
\item $SG(16, 3)  = \langle x, y \mid  x^4 = y^4 = 1, xy = y^{-1}x^{-1}, xy^{-1} = yx^{-1}\rangle$,
\end{enumerate}
then  $\mathcal{E}(\Gamma_G)=18.$
\end{corollary}
\begin{proof}
We have $|G| = 16$ and $|Z(G)| = 4$. Therefore, $\frac{G}{Z(G)} \cong {\mathbb{Z}}_2 \times {\mathbb{Z}}_2$. Hence, the result follows from Theorem \ref{cor:pp}(i).
\end{proof}
Note that all the groups considered in this section so far are non-Abelian AC-groups. We conclude this section by computing the energy of commuting graph of a finite non-Abelian AC-group in general. The following lemma is useful.
\begin{lemma}\label{sum-conjug}
 Let $C_1,\dots, C_n$ be the centralizers of non-central elements of a finite non-Abelian group $G$. Then
\[
\sum_{i=1}^{n}|C_i| = |G| + (n - 1)|Z(G)|.
\]
\end{lemma}
\begin{proof}
By \cite[Lemma 2.1]{Dut-integral} we have  $\Gamma_G= \bigsqcup_{i=1}^{n} K_{|C_i| - |Z(G)|}$. Therefore, the number of vertices in $\Gamma_G$ is given by
\[
\sum_{i=1}^{n}|C_i| - n|Z(G)| = |G| - |Z(G)|.
\]
Hence the lemma follows.
\end{proof}

\begin{theorem}\label{AC-group-energy}
Let $G$ be a finite non-Abelian AC-group with $n$ distinct centralizers of non-central elements.  Then
\[
\mathcal{E}(\Gamma_G) = 2(|G| - |Z(G)| - n).
\]
\end{theorem}
\begin{proof}
Let $ C_ 1,\dots,  C_ n$ be the distinct centralizers of non-central elements of $G$. Then by  \cite[Theorem 2.1]{Dut-integral} we have
\[
\mathcal{E}(\Gamma_G) = 2\sum_{i=1}^{n}|C_i| - 2n(|Z(G)| + 1).
\]
Hence, the result follows using Lemma \ref{sum-conjug}.
\end{proof}
\begin{corollary}
Let $A$ be any finite Abelian group and $G$   a finite non-Abelian  AC-group.  Then    the energy of the commuting graph of   $G\times A$ is given by
\[
\mathcal{E}(\Gamma_{G\times A}) =  2(|G||A| - |Z(G)||A| - n).
\]
\end{corollary}

\section*{\centerline{3. Consequences}}\setcounter{section}{3}

A finite graph $\Gamma$  is called hyperenergetic, borderenergetic and non hyperenergetic if   $\mathcal{E}(\Gamma) > \mathcal{E}(K_{|v(\Gamma)|})$, $\mathcal{E}(\Gamma) = \mathcal{E}(K_{|v(\Gamma)|})$ and $\mathcal{E}(\Gamma) < \mathcal{E}(K_{|v(\Gamma)|})$ respectively, where $v(\Gamma)$ is the set of vertices of $\Gamma$.
The study of hyperenergetic graph was initiated by Walikar et. al \cite{Walikar-99} and Gutman  \cite{Gutman-99} in 1999. However, the concept of borderenergetic  graph was introduced by Gong et al. \cite{Gong-15} in the year 2015. In this section, we shall show that the commuting graphs of the groups considered in Section 2 are neither hyperenergetic nor borderenergetic. We begin with the following result.
\begin{proposition}
If $G = M_{2mn}, D_{2m}, Q_{4m}$ and $U_{6n}$ then $\Gamma_G$ is neither hyperenergetic nor borderenergetic.
\end{proposition}
\begin{proof}
If $G = M_{2mn}$ then $|v(\Gamma_G)|= \begin{cases}
2mn - n & \text{ if $n$ is odd}\\
2mn - 2n & \text{ if $n$ is even}.
\end{cases}$

\noindent Therefore, $\mathcal{E}(K_{|v(\Gamma_G)|}) = \begin{cases}
4mn - 2n - 2& \text{ if $n$ is odd}\\
4mn - 4n - 2 & \text{ if $n$ is even}.
\end{cases}$

\noindent We have
\[
4mn - 2m - 2n - 2 < 4mn - 2n - 2 \text{ and } 4mn -4n - m - 2 < 4mn - 4n - 2.
\]
Hence, by Theorem \ref{cor2}(i), $\Gamma_{M_{2mn}}$ is neither hyperenergetic nor borderenergetic.

If $G = D_{2m}$ then  $|v(\Gamma_G)|= \begin{cases}
2m - 1 & \text{ if $m$ is odd}\\
2mn - 2 & \text{ if $m$ is even}.
\end{cases}$

\noindent Therefore, $\mathcal{E}(K_{|v(\Gamma_G)|}) = \begin{cases}
4m - 4& \text{ if $m$ is odd}\\
4m - 6 & \text{ if $m$ is even}.
\end{cases}$

\noindent We have
\[
2m - 4 < 4m - 4 \text{ and } 3m - 6 < 4m - 6.
\]
Hence, by Theorem \ref{cor2}(ii), $\Gamma_{D_{2m}}$ is neither hyperenergetic nor borderenergetic.

If $G = Q_{4m}$ then $|v(\Gamma_G)|= 4m - 2$.
Therefore, $\mathcal{E}(K_{|v(\Gamma_G)|}) = 8m - 6$.
We have
\[
6m - 3 < 8m -  6.
\]
Hence, by Theorem \ref{cor2}(iii), $\Gamma_{Q_{2m}}$ is neither hyperenergetic nor borderenergetic.

If $G = U_{6n}$ then $|v(\Gamma_G)|= 5n$.
Therefore, $\mathcal{E}(K_{|v(\Gamma_G)|}) = 10n - 2$.
We have
\[
10n - 8 < 10n -  2.
\]
Hence, by Theorem \ref{cor2}(iv), $\Gamma_{U_{6n}}$ is neither hyperenergetic nor borderenergetic.
\end{proof}
\begin{proposition}
If $G = QD_{2^n}, PSL(2, 2^k), GL(2, q), A(n, \vartheta)$ and $A(n, p)$ then $\Gamma_G$ is neither hyperenergetic nor borderenergetic.
\end{proposition}
\begin{proof}
If $G = QD_{2^n}$ then $|v(\Gamma_G)| = 2^n - 2$. Therefore, $\mathcal{E}(K_{|v(\Gamma_G)|}) = 2.2^n - 6$.
We have
\[
2^n + 2^{n - 1} - 6 < 2.2^n - 6.
\]
Hence, by Theorem \ref{cor3}(i), $\Gamma_{QD_{2^n}}$ is neither hyperenergetic nor borderenergetic.

If $G = PSL(2, 2^k)$ then $|v(\Gamma_G)| = 2^k(2^{2k} - 1)$. Therefore, $\mathcal{E}(K_{|v(\Gamma_G)|}) = 2^{3k + 1} - 2^{k + 1} - 2$.
We have
\[
2^{3k + 1} - 2^{2k + 1} - 2.2^{k + 1} - 4 < 2^{3k + 1} - 2^{k + 1} - 2.
\]
Hence, by Theorem \ref{cor3}(ii), $\Gamma_{PSL(2, 2^k)}$ is neither hyperenergetic nor borderenergetic.

If $G = GL(2, q)$ then $|v(\Gamma_G)| = (q^2 - 1)(q^2 - q) - q - 1 = q^4 - q^3 - q^2 - 1$. Therefore, $\mathcal{E}(K_{|v(\Gamma_G)|}) = 2q^4 - 2q^3 - 2q^2 - 4$.
We have
\[
2q^4 - 2q^3 - 4q^2 - 2q < 2q^4 - 2q^3 - 2q^2 - 4.
\]
Hence, by Theorem \ref{cor3}(iii), $\Gamma_{GL(2, q)}$ is neither hyperenergetic nor borderenergetic.

If $G = A(n, \vartheta)$ then $|v(\Gamma_G)| = 2^n(2^n - 1)$. Therefore, $\mathcal{E}(K_{|v(\Gamma_G)|}) = 2.2^{2n} - 2.2^n - 2$.
We have
\[
2.2^{2n} - 4.2^n + 2 < 2.2^{2n} - 2.2^n - 2.
\]
Hence, by Theorem \ref{cor3}(v), $\Gamma_{A(n, \vartheta)}$ is neither hyperenergetic nor borderenergetic.

If $G = A(n, p)$ then $|v(\Gamma_G)| = (p^{2n} - p^n)(p^n + 1)$. Therefore, $\mathcal{E}(K_{|v(\Gamma_G)|})$ $= 2.p^{3n} - 2.p^n - 2$.
We have
\[
2.p^{3n} - 4.p^n - 2 < 2.p^{3n} - 2.p^n - 2.
\]
Hence, by Theorem \ref{cor3}(vi), $\Gamma_{A(n, p)}$ is neither hyperenergetic nor borderenergetic.
\end{proof}

\begin{proposition}
If $G$ is a non-Abelian group and $|G|$ is product of two primes $p, q$  such that $p\mid (q - 1)$ then $\Gamma_G$ is neither hyperenergetic nor borderenergetic.
\end{proposition}
\begin{proof}
We have $|v(\Gamma_G)| = pq - 1$. Therefore, $\mathcal{E}(K_{|v(\Gamma_G)|}) = 2pq - 4$.
Also
\[
2pq - 2q - 3 < 2pq - 4.
\]
Hence, by Theorem \ref{cor3}(vii), the result follows.
\end{proof}

\begin{proposition}\label{quotient-prop}
 $\Gamma_G$ is neither hyperenergetic nor borderenergetic if $\frac{G}{Z(G)} \cong Sz(2), {\mathbb{Z}}_p \times {\mathbb{Z}}_p$ and $D_{2m}$.
\end{proposition}
\begin{proof}
If $\frac{G}{Z(G)} \cong Sz(2)$ then $|v(\Gamma_G)| = 19|Z(G)|$. Therefore, $\mathcal{E}(K_{|v(\Gamma_G)|}) = 38|Z(G)| - 2$.
We have
\[
28|Z(G)| - 12 < 38|Z(G)| - 2.
\]
Hence, by Theorem \ref{cor3}(iv), $\Gamma_G$ is neither hyperenergetic nor borderenergetic.

If $\frac{G}{Z(G)} \cong {\mathbb{Z}}_p \times {\mathbb{Z}}_p$ then $|v(\Gamma_G)| = |Z(G)|p^2 - |Z(G)|$. Therefore, $\mathcal{E}(K_{|v(\Gamma_G)|}) = 2|Z(G)|p^2 - 2|Z(G)| - 2$.
We have
\[
2|Z(G)|p^2 - 2|Z(G)| - p - 1 < 2|Z(G)|p^2 - 2|Z(G)| - 2.
\]
Hence, by Theorem \ref{cor:pp}(i), $\Gamma_G$ is neither hyperenergetic nor borderenergetic.

If $\frac{G}{Z(G)} \cong D_{2m}$ then $|v(\Gamma_G)| = (2m - 1)|Z(G)|$. Therefore, $\mathcal{E}(K_{|v(\Gamma_G)|}) = 2(2m - 1)|Z(G)| - 2$.
We have
\[
2(2m - 1)|Z(G)| - 2m - 2 < 2(2m - 1)|Z(G)| - 2.
\]
Hence, by Theorem \ref{cor:pp}(ii), $\Gamma_G$ is neither hyperenergetic nor borderenergetic.
\end{proof}

\begin{corollary}
Let $G$ be  a finite non-Abelian group and $p$ be any prime. Then $\Gamma_G$ is  neither hyperenergetic nor borderenergetic if
\begin{enumerate}
\item $G$ is of order $p^3$.
\item $G$ is a $4$-centralizer group.
\item $G$ is a $5$-centralizer group.
\item $G$ is a $(p + 2)$-centralizer $p$-group.
\end{enumerate}
\end{corollary}
\begin{proof}
The result follows from Proposition \ref{quotient-prop} since $\frac{G}{Z(G)}$ is isomorphic to either ${\mathbb{Z}}_p \times {\mathbb{Z}}_p$ or $D_6$.
\end{proof}

\begin{corollary}
$\Gamma_G$ is neither hyperenergetic nor borderenergetic if $G$ is given by
 ${\mathbb{Z}}_2 \times D_8$,
${\mathbb{Z}}_2 \times Q_8$,
$M_{16}$,
 ${\mathbb{Z}}_4 \rtimes {\mathbb{Z}}_4$,
 $D_8 * {\mathbb{Z}}_4$ and $SG(16, 3)$.
\end{corollary}
\begin{proof}
The result follows from Proposition \ref{quotient-prop} since $\frac{G}{Z(G)}$ is isomorphic to either ${\mathbb{Z}}_2 \times {\mathbb{Z}}_2$.
\end{proof}

In general, we  have the following theorem.
\begin{theorem}
If $G$ is a finite non-Abelian AC-group  then $\Gamma_G$ is  neither hyperenergetic nor borderenergetic.
\end{theorem}
\begin{proof}
We have $|v(\Gamma_G)| = |G| - |Z(G)|$. Therefore, $\mathcal{E}(K_{|v(\Gamma_G)|}) = 2|G| - 2|Z(G)| - 2$.
We have
\[
2|G| - 2|Z(G)| - 2n < 2|G| - 2|Z(G)| - 2.
\]
Hence, by Theorem \ref{AC-group-energy}, $\Gamma_G$ is neither hyperenergetic nor borderenergetic.
\end{proof}

We conclude this paper with a question given below.
\begin{question}
Is it true that the commuting graph of any finite non-Abelian group is neither hyperenergetic nor borderenergetic?
\end{question}

\noindent
\footnotesize{Reza Sharafdini\\
Department of Mathematics\\
Faculty of Science, Persian Gulf University,\\
Bushehr 75169-13817\\
IRAN \\
e-mail: sharafdini@pgu.ac.ir
}
\newline

\noindent
\footnotesize{Rezvan Darbandi\\
Department of Mathematics and Computer Science\\
Amirkabir University of Technology\\
Tehran, IRAN\\
e-mail: rez1.darnadi@aut.ac.ir}
\newline

\noindent
\footnotesize{Rajat Kanti Nath\\
Department of Mathematical Science\\
Tezpur University, Napaam-784028\\
Sonitpur, Assam\\
INDIA\\
e-mail: rajatkantinath@yahoo.com}

\end{document}